\documentclass[10pt]{amsart}
\usepackage{amsthm, amsfonts, amssymb, amsmath, amscd}
\usepackage{graphicx}
\usepackage{lmodern}
\usepackage[margin=3.5cm]{geometry}

\newcommand{\To}{\rightarrow}

\newcommand{\e}{\epsilon}

\newcommand{\C}{\mathbb{C}}
\newcommand{\R}{\mathbb{R}}

\newcommand{\Z}{\mathbb{Z}}
\newcommand{\N}{\mathbb{N}}
\newcommand{\F}{\mathbb{F}}
\newcommand{\Cstar}{\mathrm{C^*}}
\newcommand{\Cred}{\mathrm{C^*_r}}
\newcommand{\G}{\Gamma}
\newcommand{\rd}{\mathrm{rd}}

\newcommand{\RD}{\mathrm{RD}}
\newenvironment{myindentpar}[1]%
{\begin{list}{}%
         {\setlength{\leftmargin}{#1}}%
         \item[]%
}
{\end{list}}

\newtheorem*{theorem}{Theorem}
\newtheorem{thm}{Theorem}[section]
\newtheorem{lem}[thm]{Lemma}

\newtheorem{prop}[thm]{Proposition}

\theoremstyle{definition}

\newtheorem*{prob}{Problem}

\newtheorem*{ack}{Acknowledgments}

\begin{document}

\title{On the degree of rapid decay}
\author{Bogdan Nica}
\address{Department of Mathematics, Vanderbilt University, Nashville, TN 37240, USA}
\curraddr{Department of Mathematics and Statistics, University of Victoria, Victoria (BC), Canada V8W 3R4}
\date{\today}
\begin{abstract}
A finitely generated group $\G$ equipped with a word-length is said to satisfy property RD if there are $C, s\geq 0$ such that for all $n\geq 1$, we have $\|a\|\leq C (1+n)^s \|a\|_2$ whenever $a\in\C\G$ is supported on elements of length at most $n$. 

We show that, for infinite $\G$, the degree $s$ is at least $1/2$. 
\end{abstract}
\maketitle

\section{Introduction}
Let $\G$ be a finitely generated group, and fix a word-length on $\G$. We say that $\G$ has \emph{property RD} if there exist $C,s\geq 0$ such that, for all integers $n\geq 0$, we have $\|a\|\leq C (1+n)^s \|a\|_2$ whenever $a\in\C\G$ is supported on elements of length at most $n$. Here $\|\cdot\|$ denotes the operator norm coming from the regular representation of $\G$ on $\ell^2\G$. 

This property originates from Haagerup's seminal paper \cite{Haa79}, where it is shown that free groups have what we now call property RD. The explicit definition of property RD is due to Jolissaint \cite{Jol90}; a result from \cite{Jol90} we would like to quote here is the fact that groups of polynomial growth have property RD. As further examples of groups satisfying property RD, we mention: hyperbolic groups (\cite{dHa88}) and, more generally, groups that are relatively hyperbolic to subgroups having property RD (\cite{DS05}); groups acting freely on finite dimensional CAT(0) cube complexes (\cite{CR05}); cocompact lattices in $\mathrm{SL}_3(\F)$ for $\F$ a local field (\cite{RRS98}) and for $\F=\R,\C$ (\cite{Laf00}); mapping class groups (\cite{BM08}).

Property RD for a group $\G$ is relevant to the study of its reduced $\Cstar$-algebra $\Cred\G$. The first significant use of property RD is the proof by Connes and Moscovici \cite{CM90} that hyperbolic groups satisfy the Novikov conjecture. A related K-theoretic application features in the remarkable work of Lafforgue \cite{Laf02}, leading eventually to the proof that hyperbolic groups satisfy the Baum-Connes conjecture (\cite{MY02}, \cite{Laf02}). In another direction, Property RD is used in \cite{DdH99} to show that $\Cred\G$ has stable rank $1$ whenever $\G$ is a torsion-free, non-elementary hyperbolic group.

In this note, we are interested in quantifying property RD. For $s\geq 0$ consider the following property:
\begin{itemize}
\item[$(\RD^s_\bullet)$] there is $C\geq 0$ such that, for all integers $n\geq 0$, we have $\|a\|\leq C (1+n)^s \|a\|_2$ whenever $a\in\C\G$ is supported on elements of length at most $n$
\end{itemize}
Note that the satisfaction of $(\RD^s_\bullet)$ does not depend on the choice of word-length for $\G$.

Finite groups satisfy $(\RD^0_\bullet)$. Conversely, infinite groups cannot satisfy $(\RD^0_\bullet)$; this was first proved by Rajagopalan in \cite{Raj63}. Rajagopalan's result was a partial positive answer towards the ``$L^p$-conjecture'' that he formulated around 1960: if $G$ is a locally compact group, and $L^p(G)$ is closed under convolution for some $p\in (1,\infty)$, then $G$ is a compact. The complete resolution of the $L^p$-conjecture is due to Saeki \cite{Sae90}, almost 30 years after its formulation.

The following question then arises: if an infinite group satisfies $(\RD^s_\bullet)$, how small can $s$ be? The main result of this note answers this question:

\begin{theorem}
Let $\G$ be an infinite, finitely generated group. If $\G$ satisfies $(\RD^s_\bullet)$ then $s \geq \frac{1}{2}$. 
\end{theorem}

This result is sharp, since $(\RD^s_\bullet)$ for $s=\frac{1}{2}$ is satisfied by all virtually-$\Z$ groups.

We do not know how to dismiss the main theorem as being trivial, so we settle for the next best thing: an elementary proof (given in Section~\ref{infinite}), inspired by Saeki's solution to the $L^p$-conjecture. Rajagopalan's proof from \cite{Raj63}, using structural results on so-called $H^*$-algebras, does not seem flexible enough to be adaptable here. 

It should be pointed out that the main theorem is interesting for infinite torsion groups only, much like the $L^p$-conjecture. Indeed, if $\G$ contains an infinite cyclic subgroup then the main result is a mere observation (cf. Prop.~\ref{poly} and Prop.~\ref{her}). We are therefore accounting here for the possibility that infinite torsion groups with property RD might exist - a possibility which seems completely open at this time.

\begin{ack} I thank Gennadi Kasparov and Guoliang Yu for financial support during the summer of 2009. I also thank the referee for some useful comments.
\end{ack}

\section{Preliminaries}
Throughout this paper, groups are assumed to be finitely generated and equipped with a word-length. The choice of word-length is irrelevant for the discussion herein.

\subsection{Notation}For two numerical functions $f,g:\N\To\N$, we write $f(n)\preccurlyeq g(n)$ to mean that there is $C\geq 0$ such that $f(n)\leq Cg(n)$ for all $n\in \N$; we write $f(n) \asymp g(n)$ whenever we have both $f(n)\preccurlyeq g(n)$ and $g(n)\preccurlyeq f(n)$, i.e., there are $C\geq c\geq 0$ such that $cg(n)\leq f(n)\leq Cg(n)$ for all $n\in \N$.

\subsection{A reformulation of property RD}\label{version} An equivalent definition of property RD is the following: a group $\G$ is said to satisfy property RD if there are $C,s\geq 0$ such that $\|a\|\leq C\|a\|_{2,s}$ for all $a\in\C\G$, where
\[\big\|\sum a_g g\big\|_{2,s}:= \big\|\sum a_g(1+|g|)^s g\big\|_2=\sqrt{\sum |a_g|^2(1+|g|)^{2s}}.\]
Correspondingly, we quantify by considering, for $s\geq 0$, the following property:
\begin{itemize}
\item[$(\RD^s)$] there is $C \geq 0$ such that $\|a\|\leq C\|a\|_{2,s}$ for all $a\in\C\G$
\end{itemize}
Informally, $(\RD^s)$ is a ``de-localized'' version of $(\RD_\bullet^s)$ (recall, the latter is a property defined with reference to balls - hence the label $\bullet$). The equivalence between the two formulations of property RD is well-known. The next lemma records this equivalence in a precise fashion: 

\begin{lem}\label{laf}
We have that $(\RD^s)$ implies $(\RD^s_\bullet)$, and that $(\RD^s_\bullet)$ implies $(\RD^{s+\e})$ for all $\e>0$. 
\end{lem}

\begin{proof}
The first implication is obvious. The second implication is essentially contained in the proof of Proposition 1.2a) in \cite{Laf00}, and works as follows. Assume that $\G$ satisfies $(\RD^{s}_\bullet)$. For $n\geq 0$ we let $A_n$ be the annulus $\{g\in \G:\: 2^n-1\leq |g|<2^{n+1}-1\}$. Then for $a=\sum a_g g\in\C\G$ we have:
\begin{eqnarray*}
\big\|\sum a_g g\big\| &\leq& \sum_{n\geq 0}\bigg\|\sum_{g\in A_n} a_g g\bigg\|\stackrel{^{(\RD^{s}_\bullet)}}{\leq} C\sum_{n\geq 0}2^{s(n+1)}\bigg\|\sum_{g\in A_n} a_g g\bigg\|_2\\
&\leq&C \sum_{n\geq 0}2^{s(n+1)-(s+\e)n}\bigg\|\sum_{g\in A_n} a_g g\bigg\|_{2,s+\e}= 2^sC \sum_{n\geq 0}2^{-\e n}\bigg\|\sum_{g\in A_n} a_g g\bigg\|_{2,s+\e}\\
&\leq& 2^sC\bigg(\sum_{n\geq 0}2^{-2\e n}\bigg)^\frac{1}{2} \bigg(\sum_{n\geq 0} \bigg\|\sum_{g\in A_n} a_g g\bigg\|_{2,s+\e}^2\bigg)^\frac{1}{2}=C'(s,\e) \big\|\sum a_g g\big\|_{2,s+\e}
\end{eqnarray*}
We conclude that $\G$ satisfies $(\RD^{s+\e})$ for each $\e>0$.
\end{proof}

In general, $(\RD^s_\bullet)$ does not imply $(\RD^s)$; see the following example.

\subsection{Groups of polynomial growth}%\label{polynomial} 
Let $\G$ be a group of polynomial growth. Then its growth function $\gamma$ satisfies $\gamma(n) \asymp n^{d(\G)}$ for some integer $d(\G)$ called the \emph{degree of growth} of $\G$ (see \cite[VII.26 \& VII.29]{dHa00}). With this notion at hand, we can state the following explicit description:

\begin{prop}\label{poly}
Let $\G$ be a group of polynomial growth. Then:
 \begin{enumerate}
 \item $\G$ satisfies $(\RD^s_\bullet)$ if and only if $s\geq \frac{1}{2}d(\G)$;
 \item $\G$ satisfies $(\RD^s)$ if and only if $s> \frac{1}{2}d(\G)$.
\end{enumerate}
\end{prop}  

\begin{proof} The following fact quantifies a result of Jolissaint \cite[Cor.3.1.8]{Jol90} saying that the amenable groups which enjoy property RD are precisely the polynomial growth groups:

\smallskip
\noindent\textbf{Fact.} Let $\G$ be amenable, with growth function $\gamma$. Then $\G$ satisfies $(\RD^s_\bullet)$ if and only if $\gamma(n)\preccurlyeq n^{2s}$.

\begin{myindentpar}{.5cm}
\emph{Proof:} If $\G$ satisfies $(\RD^s_\bullet)$ then $\|\chi(B_n)\|\leq C (1+n)^s\sqrt{\gamma(n)}$, where $\chi(B_n)$ is the characteristic function of the $n$-ball of $\G$. Since $\G$ is amenable, we have that $\|a\|=\|a\|_1$ for all $a\in\C\G$ with positive coefficients; for $\chi(B_n)$, we get $\|\chi(B_n)\|=\gamma(n)$. It follows that $\gamma(n)\leq C^2 (1+n)^{2s}$.

Conversely, assume that $\gamma(n)\leq C n^{2s}$ for some $C>0$. For $a\in\C\G$ supported on elements of length at most $n$, we have $\|a\|\leq \|a\|_1\leq \sqrt{\gamma(n)}\|a\|_2\leq \sqrt{C}(1+n)^s\|a\|_2$. Thus $\G$ satisfies $(\RD^s_\bullet)$.  
\end{myindentpar}

\smallskip
\noindent The above fact yields Part (1). Most of Part (2) follows by combining Part (1) with Lemma~\ref{laf}; the only thing left to check is that $\G$ does not satisfy $(\RD^s)$ at the critical value $s=\frac{1}{2}d(\G)$. Arguing by contradiction, let us assume that $\G$ satisfies $(\RD^s)$ for $s=\frac{1}{2}d(\G)$. Fix an integer $N\geq 0$ and consider 
\[a_N=\sum_{1\leq n\leq N} (1+n)^{-d(\G)}\chi(S_n)\in\C\G\]
where $\chi(S_n)$ denotes the characteristic function of the $n$-sphere of $\G$. We have $\|a_N\|\leq C\|a_N\|_{2,s}$ where $C\geq 0$ is independent of $N$. Since
\[\|a_N\|=\sum_{1\leq n\leq N} (1+n)^{-d(\G)}|S_n|, \qquad \|a_N\|_{2,s}=\sqrt{\sum_{1\leq n\leq N} (1+n)^{-d(\G)}|S_n|}\]
we infer that $\sum (1+n)^{-d(\G)}|S_n|$ converges. This is absurd, since $|S_n|\succcurlyeq n^{d(\G)-1}$ (see \cite[VII.32]{dHa00}).
\end{proof}

\subsection{Heredity} We make the following observation, whose easy proof is left to the reader.

\begin{prop}\label{her} Let $\G'$ be a subgroup of $\G$.
\begin{enumerate}
\item If $\G$ satisfies $(\RD_\bullet^s)$ then $\G'$ satisfies $(\RD_\bullet^s)$;
\item Assume that $\G'$ has finite index in $\G$. Then $\G$ satisfies $(\RD_\bullet^s)$ if and only if $\G'$ satisfies $(\RD_\bullet^s)$.
\end{enumerate}
\end{prop}

This proposition also holds for $(\RD^s)$ instead of $(\RD_\bullet^s)$.

%%%%%%%%%%%%%%%%%%%%%%%%%%%%%%%%%
\section{Proof of Main Theorem}\label{infinite}
This section contains the proof of our main result. We start off with two lemmas which are free of any RD assumption. In what follows, products of functions on $\G$ are convolution products, and inequalities are in the pointwise sense.

\begin{lem}\label{one}
For all $n,k\geq 1$ we have $\chi(B_n)\chi(B_{n+k})\geq |B_n|\chi(B_k)$.
\end{lem}

\begin{proof} The coefficient of each $h\in B_k$ is at least $|B_n|$, since $g^{-1}h\in B_{n+k}$ whenever $g\in B_n$.
\end{proof}

For $r\geq 1$ and $\alpha >0$, consider the following formal sum:
\[Z_r(\alpha)=\sum_{k\geq 1} k^{-\alpha}\;\frac{\chi(B_{rk})}{\|\chi(B_{rk})\|_2}\] 

\begin{lem}\label{two}
Let $\alpha, \beta>0$ with $\alpha+\beta >1$. Then $Z_r(\alpha)Z_r(\beta)\geq\frac{1}{\alpha+\beta-1}\; Z_r(\alpha+\beta-1)$.
\end{lem}

\begin{proof} Lemma~\ref{one}, together with the obvious inequality $\|\chi(B_{r(j+k)})\|_2\leq \|\chi(B_{rj})\|_2\|\chi(B_{rk})\|_2$, give 
\[\frac{\chi(B_{rk})}{\|\chi(B_{rk})\|_2}\;\frac{\chi(B_{r(j+k)})}{\|\chi(B_{r(j+k)})\|_2}\geq \frac{\|\chi(B_{rk})\|^2_2\;\chi(B_{rj})}{\|\chi(B_{rk})\|_2\|\chi(B_{r(j+k)})\|_2}\geq \frac{\chi(B_{rj})}{\|\chi(B_{rj})\|_2}\]
for all $j,k\geq 1$.
Therefore:
\begin{eqnarray*}
Z_r(\alpha)Z_r(\beta)
&\geq&  \sum_{j\geq 1} \sum_{k\geq 1} k^{-\alpha}(j+k)^{-\beta}\frac{\chi(B_{rk})}{\|\chi(B_{rk})\|_2}\;\frac{\chi(B_{r(j+k)})}{\|\chi(B_{r(j+k)})\|_2}\\
&\geq& \sum_{j\geq 1} \Big(\sum_{k\geq 1} (j+k)^{-\alpha-\beta}\Big)\frac{\chi(B_{rj})}{\|\chi(B_{rj})\|_2}
\end{eqnarray*}
Since for each $j \geq 1$ we have 
\[\sum_{k\geq 1} (j+k)^{-(\alpha+\beta)}>\int_j^\infty x^{-(\alpha+\beta)}\: dx=\frac{j^{-(\alpha+\beta-1)}}{\alpha+\beta-1},\]
the desired inequality follows. \end{proof}

We now come to the proof of

\begin{thm}\label{main} If $\G$ is an infinite group satisfying $(\RD^s_\bullet)$, then $s \geq \frac{1}{2}$. 
\end{thm}

\begin{proof}
Assume, by way of contradiction, that $\G$ satisfies $(\RD^s_\bullet)$ for some $s\in[0,\frac{1}{2})$. Using Lemma~\ref{one} and $(\RD^s_\bullet)$, we have
\begin{eqnarray*}
\|\chi(B_r)\|_2^2\|\chi(B_{rk})\|_2&=&\||B_r|\chi(B_{rk})\|_2\leq \|\chi(B_r)\chi(B_{r(k+1)})\|_2\\
&\leq&\|\chi(B_r)\|\|\chi(B_{r(k+1)})\|_2 \leq C (1+r)^s\|\chi(B_r)\|_2\|\chi(B_{r(k+1)})\|_2
\end{eqnarray*}
that is
\[\|\chi(B_{r(k+1)})\|_2\geq \frac{\|\chi(B_r)\|_2}{C (1+r)^s}\|\chi(B_{rk})\|_2\]
for all $r,k\geq 1$. Fix $r\geq 1$ such that $\|\chi(B_r)\|_2\geq 2C (1+r)^s$ - a choice made possible by the fact that the volume growth of $\G$ is at least linear. Then $\|\chi(B_{r(k+1)})\|_2\geq 2\|\chi(B_{rk})\|_2$ for all $k\geq 1$.

Next, we claim that $Z_r(\alpha)\in\ell^2\G$ if and only if $\alpha>\frac{1}{2}$. To show this, we compare $\|Z_r(\alpha)\|_2^2$ against $\sum k^{-2\alpha}$. One bound holds in general: 
\[\|Z_r(\alpha)\|_2^2\geq\sum_{k=l\geq 1} k^{-\alpha}l^{-\alpha}\Big\langle \frac{\chi(B_{rk})}{\|\chi(B_{rk})\|_2}, \frac{\chi(B_{rl})}{\|\chi(B_{rl})\|_2}\Big\rangle=\sum_{k\geq 1} k^{-2\alpha}\]
For the other bound, we write:
\[\|Z_r(\alpha)\|_2^2\leq 2\sum_{l\geq k\geq 1} k^{-\alpha}l^{-\alpha}\Big\langle \frac{\chi(B_{rk})}{\|\chi(B_{rk})\|_2}, \frac{\chi(B_{rl})}{\|\chi(B_{rl})\|_2}\Big\rangle=2\sum_{k\geq 1} \bigg(\sum_{l\geq k}\frac{l^{-\alpha}}{\|\chi(B_{rl})\|_2}\bigg) k^{-\alpha}\|\chi(B_{rk})\|_2\]
From $\|\chi(B_{r(l+1)})\|_2\geq 2\|\chi(B_{rl})\|_2$ we deduce that 
\[\frac{(l+1)^{-\alpha}}{\|\chi(B_{r(l+1)})\|_2}<\frac{1}{2}\frac{l^{-\alpha}}{\|\chi(B_{rl})\|_2}\] 
hence
\[\sum_{l\geq k}\frac{l^{-\alpha}}{\|\chi(B_{rl})\|_2}< 2 \frac{k^{-\alpha}}{\|\chi(B_{rk})\|_2}\] 
which, in turn, implies $\displaystyle \|Z_r(\alpha)\|_2^2<4 \sum_{k\geq 1} k^{-2\alpha}$. The claim is proved.

Pick $t$ such that $s<t<\frac{1}{2}$. Also, pick $\alpha, \beta$ such that the following are satisfied:
\[\alpha>\frac{1}{2}+t, \qquad \beta>\frac{1}{2}, \qquad \alpha+\beta -1\leq \frac{1}{2} \]
On one hand, $Z_r(\alpha+\beta-1)$ is not in $\ell^2\G$. On the other hand, $Z_r(\alpha-t)$ and $Z_r(\beta)$ are in $\ell^2\G$. From 
\[\|Z_r(\alpha)\|_{2,t}\leq \bigg\|\sum_{k\geq 1} k^{-\alpha}(1+rk)^t\frac{\chi(B_{rk})}{\|\chi(B_{rk})\|_2}\bigg\|_2\leq \bigg\|\sum_{k\geq 1} k^{-\alpha}(2rk)^t\frac{\chi(B_{rk})}{\|\chi(B_{rk})\|_2}\bigg\|_2=(2r)^t\|Z_r(\alpha-t)\|_2\]
and the fact that $(\RD^t)$ holds (Lemma~\ref{laf}), we deduce that $Z_r(\alpha)$ is a bounded operator on $\ell^2\G$. Then $Z_r(\alpha)Z_r(\beta)$ is in $\ell^2\G$, so $Z_r(\alpha+\beta-1)$ is in $\ell^2\G$ by Lemma~\ref{two}. This contradiction ends the proof. \end{proof}

\section{Final remarks}
According to Lemma~\ref{laf}, we have $\inf \{s:\; \textrm{$\G$ satisfies $(\RD^s)$}\}=\inf \{s:\; \textrm{$\G$ satisfies $(\RD^s_\bullet)$}\}$; we denote this quantity by  $\rd(\G)$ and we call it the \emph{RD-degree} of $\G$. By definition, $\rd(\G)$ is finite precisely when $\G$ has property RD. Observe that $\rd(\G)$ is independent of the choice of word-length for $\G$.

In terms of the RD-degree, our discussion can be summarized as follows:

$\cdot$ if $\G$ is infinite then $\rd(\G)\geq \frac{1}{2}$; if $\G$ is finite then $\rd(\G)=0$

$\cdot$ if $\G$ has polynomial growth then $\rd(\G)=\frac{1}{2}d(\G)$, where $d(\G)$ denotes the growth degree of $\G$

$\cdot$ if $\G'$ is a subgroup of $\G$ then $\rd(\G')\leq\rd(\G)$

$\cdot$ if $\G'$ and $\G$ are commensurable then $\rd(\G')=\rd(\G)$

\noindent These properties suggest that $2\: \rd(\cdot)$ can be thought of as a dimension function on finitely generated groups. Computing the RD-degree of other groups would be, of course, interesting in this regard. We single out the following 
\begin{prob} Compute $\rd(\mathrm{F}_r)$, where $\mathrm{F}_r$ denotes the free group of rank $r\geq 2$.
\end{prob}
Note that the answer is independent of $r$. It is known that $1\leq \rd(\mathrm{F}_r)\leq \frac{3}{2}$ (Haagerup's estimates from \cite{Haa79} yield the upper bound; the lower bound is a consequence of Cohen's computations from \cite{Coh82}).

We close by reminding the reader that the following question is open: are there infinite, finitely generated torsion groups which enjoy property RD?

\end{document}